\documentclass[oneside,11pt]{amsart}

\usepackage{float}
\usepackage{amscd,amssymb}
\usepackage[mathscr]{eucal}
\usepackage[all]{xy}
\usepackage{mathrsfs}
\usepackage{amsfonts}
\usepackage{amsmath}
\usepackage{amsthm}
\usepackage{latexsym}
\usepackage[all]{xy}
 
\title{A derived category approach to Kempf's vanishing theorem}

\author{Alexander Samokhin}

\address{Math. Institut, Heinrich-Heine-Universit\"at, D-40204 D\"usseldorf, Germany}
\address{\it and}
\address{Institute for Information Transmission Problems, Moscow,  Russia}
\email{alexander.samokhin@gmail.com}

\newcommand{\Oo}{\mathcal O}

\newcommand{\Pp}{\mathbb P}

\newcommand{\Ff}{\mathcal F}
\newcommand{\Ee}{\mathcal E}
\newcommand{\Ll}{\mathcal L}

\newcommand{\sk}{\sf k}

\newcommand*{\Hom}{\mathop{\mathrm Hom}\nolimits}
\newcommand*{\Dd}{\mathop{\mathrm D\kern0pt}\nolimits}

\newtheorem{theorem}{Theorem}[section]
\newtheorem{corollary}{Corollary}[section]
\newtheorem{lemma}{Lemma}[section]
\newtheorem{proposition}{Proposition}[section]

\newtheorem{remark}{Remark}[section]
\newtheorem{definition}{Definition}[section]

\long\def\comment#1{}

\begin{document}

\maketitle 

\begin{abstract}
We give a proof of the Andersen--Haboush identity 
(cf. \cite{Ann}, \cite{Hab}) that implies Kempf's vanishing theorem. Our argument is based on the  structure of derived categories of coherent sheaves on flag varieties over $\mathbb Z$.
\end{abstract}

\vspace{0.3cm}

\section{\bf Introduction}

\vspace*{0.3cm}

Let $\bf G$ be a split semisimple simply connected algebraic group over a perfect field $\sk$ of characteristic $p$. The weight $-\rho = - \sum \omega _i$, where $ \omega _i$ are  the fundamental weights of $\bf G$, is known to play a fundamental r\^ole in representation theory of $\bf G$. For $q=p^n, n\geq 1$, the Steinberg weight $(q-1)\rho$ is equally important in representation theory of semisimple groups in defining characteristic. In particular, there is a remarkable property that the corresponding line bundle $\Ll _{(q-1)\rho}$ on the flag variety ${\bf G}/{\bf B}$ enjoys: its pushforward under the $n$--th iteration of Frobenius morphism is a trivial vector bundle whose space of global sections is canonically identified with the Steinberg representation ${\sf St}_q$:


\begin{equation}\label{eq:Steinberg}
{\sf F}^n_{\ast}\Ll _{(q-1)\rho} = {\sf St}_q\otimes \Oo _{{\bf G}/{\bf B}}.
\end{equation}

\vspace*{0.2cm}

This was proven independently and at around the same time by Andersen in \cite{Ann} and by Haboush in \cite{Hab}. Back to the weight $-\rho$, isomorphism of vector bundles (\ref{eq:Steinberg}) is equivalent to saying that the line bundle $\Ll _{-\rho}$ is an "eigenvector" with respect to Frobenius morphism, i.e. ${\sf F}^n_{\ast}\Ll _{-\rho} = {\sf St}_q\otimes \Ll _{-\rho}$. This fact has many important consequences for representation theory of algebraic groups in characteristic $p$: in particular, the Kempf vanishing theorem \cite{Kem} easily follows from it (see \cite{Ann} and \cite{Hab}). The proofs of {\it loc.cit.} were essentially representation--theoretic. The goal of this note is to
prove isomorphism (\ref{eq:Steinberg}) using the structure of the derived category of coherent sheaves on the flag variety ${\bf G}/{\bf B}$. In a nutshell, the idea is as follows.
Given a smooth algebraic variety $X$ and a semiorthogonal decomposition $\langle \sf D_0,\sf D_1\rangle$ of the derived category $\Dd ^b(X)$ (see Section \ref{sec:Prelim} for the details), any object of $\Dd ^b(X)$ -- in particular, any vector bundle $\Ff$ on $X$, can be decomposed with respect to $\sf D_0$ and $\sf D_1$. Thus, if $\Ff$ is right orthogonal to $\sf D_1$, i.e. $\Hom _{\Dd ^b(X)}^{\cdot}(\sf D _1,\Ff)=0$, it automatically belongs to $\sf D_0$. It turns out that for a semiorthogonal decomposition of the derived category $\Dd ^b({\bf G}/{\bf B})$ into two pieces, one of which is the admissible subcategory $\langle \Ll _{-\rho}\rangle$ generated by the single line bundle $\Ll _{-\rho}$ and the other one being its left orthogonal $^{\perp}\langle \Ll _{-\rho}\rangle$, the bundle ${\sf F}^n_{\ast}\Ll _{-\rho}$ is right orthogonal to $^{\perp}\langle \Ll _{-\rho}\rangle$. Therefore, it should belong to the subcategory  $\langle \Ll _{-\rho}\rangle$. Being generated by a single exceptional bundle, the latter subcategory is equivalent to the derived category of vector spaces over $\sk$; thus, one has ${\sf F}^n_{\ast}\Ll _{-\rho}=\Ll _{-\rho}\otimes {\sf V}$ for some graded vector space $\sf V$. Since the left hand side of this isomorphism is a vector bundle, i.e. a pure object of $\Dd ^b({\bf G}/{\bf B})$, the graded vector space $\sf V$ should only have a non--trivial zero degree part, which is a vector space of dimension $q^{\rm dim({\bf G}/{\bf B})}$. Tensoring the both sides with $\Ll _{\rho}$ and taking the cohomology, one obtains an isomorphism ${\sf V} = {\rm H}^0({\bf G}/{\bf B},\Ll _{(q-1)\rho})={\sf St}_q$, and hence isomorphism (\ref{eq:Steinberg}).

Unfolding this argument takes the rest of the note. 
The key step consists of proving a special property of the semiorthogonal decomposition described above that allows to easily check  the orthogonality properties for the bundle ${\sf F}^n_{\ast}\Ll _{-\rho}$. 
This is done in Section \ref{sec:main}. Theorem \ref{th:Haboushth}, which is equivalent to isomorphism (\ref{eq:Steinberg}), immediately follows from it.

The present note was initially motivated by the author's computations of Frobenius pushforwards of homogeneous vector bundles on flag varieties. The derived localization theorem of \cite{BMR} implies, in particular, that for a regular weight $\chi$ (that is, for a weight having trivial stabilizer with respect to the dot--action of the (affine) Weyl group)
the bundle ${\sf F}_{\ast}\Ll _{\chi}$ is a generator in the derived category $\Dd ^b({\bf G}/{\bf B})$; in other words, there are sufficiently many  indecomposable summands of 
${\sf F}_{\ast}\Ll _{\chi}$ to generate the whole derived category $\Dd ^b({\bf G}/{\bf B})$. Knowing indecomposable summands of these bundles (e.g., for $p$--restricted weights) may clarify, in particular, cohomology vanishing patterns of line bundles on ${\bf G}/{\bf B}$. On the contrary, the weight $-\rho$ being the most singular, the thick category generated by the bundle ${\sf F}_{\ast}\Ll _{-\rho}$ "collapses" to the subcategory generated by the single line bundle $\Ll _{-\rho}$, which is encoded in isomorphism (\ref{eq:Steinberg}).

\subsection*{\bf Acknowledgements}
We are indebted to Roman Bezrukavnikov, Michel Brion, Jim Humphreys, Nicolas Perrin, and Alexander Polishchuk for their advice and valuable suggestions, and to Michel Van den Bergh for his interest in this work. The author gratefully acknowledges support from the strategic research fund of the Heinrich-Heine-Universit\"at D\"usseldorf (grant SFF F-2015/946-8).
We would also like to thank the ICTP, Trieste for providing excellent working facilities and the Deutsche Bahn for a truly creative atmosphere on their lonely night IC trains in the Rheinland-Pfalz region.

\subsection*{Notation}
Given a split semisimple simply connected algebraic group $\bf G$ over a perfect field $k$, let $\bf T$ denote a maximal torus of $\bf G$, and let ${\bf T}\subset {\bf B}$ be a Borel subgroup containing $\bf T$. The flag variety of Borel subgroups in $\bf G$ is denoted ${\bf G/B}$. Denote ${\rm X}({\bf T})$ the weight lattice, and let $\rm R$ and  $\rm R ^{\vee}$ denote the root and coroot lattices, respectively. The Weyl group ${\mathcal W}={\rm N}({\bf T})/{\bf T}$ acts on $X({\bf T})$ via the dot--action: if $w\in {\mathcal W}$, and 
$\lambda \in {\rm X}({\bf T})$, then $w\cdot \lambda = w(\lambda + \rho) - \rho$, where $\rho$ is the sum of fundamental weights. Let $\rm S$ be the set of simple roots relative to the choice of a Borel subgroup than contains $\bf T$. A parabolic subgroup of $\bf G$ is usually denoted by $\bf P$; in particular, for a simple root $\alpha \in \rm S$, denote ${\bf P}_{\alpha}$ the minimal parabolic subgroup of $\bf G$ associated to $\alpha$.
Given a weight $\lambda \in  {\rm X}({\bf T})$, denote $\Ll _{\lambda}$ the corresponding line bundle on ${\bf G}/{\bf B}$. Given a morphism $f:X\rightarrow Y$ between two schemes, we write $f_{\ast},f^{\ast}$ for the corresponding derived functors of push--forwards and pull--backs.

\vspace*{0.3cm}

\section{\bf  Some preliminaries}\label{sec:Prelim}

\vspace*{0.3cm}

\subsection{\bf Flag varieties of Chevalley groups over ${\mathbb Z}$}

Let ${\mathbb G}\rightarrow \mathbb Z$ be a semisimple Chevalley group scheme (a smooth affine group scheme over ${\rm Spec}(\mathbb Z)$ whose geometric fibres are connected semisimple algebraic groups), and ${\mathbb G}/{\mathbb B}\rightarrow \mathbb Z$ be the corresponding Chevalley flag scheme (resp., 
the corresponding parabolic subgroup scheme ${\mathbb G}/{\mathbb P}\rightarrow \mathbb Z$ for a standard parabolic subgroup scheme ${\mathbb P}\subset \mathbb G$ over ${\mathbb Z}$). Then ${\mathbb G/\mathbb P}\rightarrow {\rm Spec}({\mathbb Z})$ is flat and any line bundle $\Ll$ on ${\bf G/P}$ also comes from a line bundle $\mathbb L$ on ${\mathbb G/\mathbb P}$. Let $\sk$ be a field of arbitrary characteristic, and ${\bf G/B}\rightarrow {\rm Spec}(\sk)$ be the flag variety obtained by base change along ${\rm Spec}(\sk)\rightarrow {\rm Spec}(\mathbb Z)$.

\vspace*{0.1cm}

\subsection{\bf Cohomology of line bundles on flag varieties}\label{subsec:cohlinbunflags}

We recall first the classical Bott's theorem (see \cite{Dem}). Let ${\mathbb G}\rightarrow \mathbb Z$ be a semisimple Chevalley group scheme as above. Assume given a weight $\chi \in X({\mathbb T})$, and let $\Ll _{\chi}$ be the corresponding line bundle on ${\mathbb G}/{\mathbb B}$. The weight $\chi$ is called {\it singular}, if it lies on a wall of some Weyl chamber defined by $\langle -, \alpha ^{\vee}\rangle =0$ for some coroot $\alpha ^{\vee}\in {\rm R}^{\vee}$. Weights, which are not singular, are called {\it regular}. A weight $\chi$ such that $\langle \chi ,\alpha ^{\vee}\rangle \geq 0$ for all simple coroots $\alpha ^{\vee}$ is called {\it dominant}. Let $\sk$ be a field of characteristic zero, and ${\bf G}/{\bf B}\rightarrow {\rm Spec}(\sk)$ the corresponding flag variety over $\sk$. The weight $\chi \in X({\bf T})$ defines a line bundle $\Ll _{\chi}$ on 
${\bf G}/{\bf B}$.

\vspace*{0.2cm}

\begin{theorem}\cite[Theorem 2]{Dem}\label{th:Bott-Demazure_th}

\vspace*{0.2cm}

\begin{itemize}

\vspace*{0.2cm}

\item[(a)] If $\chi +\rho$ is singular, then ${\rm H}^i({\bf G}/{\bf B},\Ll _{\chi})= 0$ for all $i$.

\vspace*{0.2cm}

\item[(b)] If If $\chi + \rho$  is regular and dominant, then ${\rm H}^i({\bf G}/{\bf B},\Ll _{\chi}) = 0$ for $i>0$.

\vspace*{0.2cm}

\item[(c)]  If $\chi + \rho$  is regular, then ${\rm H}^i({\bf G}/{\bf B},\Ll _{\chi})\neq 0$ for the unique degree $i$, which is equal to $l(w)$. Here $l(w)$ is the length of an element $w$ of the Weyl group that takes $\chi$ to the dominant chamber, i.e. $w\cdot \chi \in X_{+}({\bf T})$. The cohomology group ${\rm H}^{l(w)}({\bf G}/{\bf B},\Ll _{\chi})$ is the irreducible $\bf G$--module of highest weight $w\cdot \chi$.

\end{itemize}

\end{theorem}

\vspace*{0.2cm}

\begin{remark}\label{rem:Demazure_bits_over_Z}
{\rm Some bits of Theorem \ref{th:Bott-Demazure_th} are still true over $\mathbb Z$: if a weight
$\chi$ is such that $\langle \chi + \rho, \alpha ^{\vee}\rangle =0$ for some simple root $\alpha$, then the corresponding line bundle is acyclic. Indeed, Lemma from 
\cite[Section 2]{Dem} holds over fields of arbitrary characteristic. Besides this, however, very little of Theorem \ref{th:Bott-Demazure_th} holds over $\mathbb Z$ (see \cite[Part II, Chapter 5]{Jan}).}
\end{remark}

\vspace*{0.1cm}

From now on, unless specified otherwise, the base field $\sf k$ is assumed to be a perfect field of characteristic $p>0$. 

\subsection{\bf Kempf's vanishing theorem}\label{subsec:Kempf_vanishing}
 
Kempf's vanishing theorem, originally proven by Kempf in \cite{Kem}, and subsequently by Andersen \cite{Ann} and Haboush \cite{Hab} with shorter representation--theoretic proofs (see also \cite[Part II, Chapter 4]{Jan}), states that given a dominant weight $\chi \in X(\bf T)$, the cohomology groups ${\rm H}^i({\bf G}/{\bf B},\Ll _{\chi})$ vanish in positive degrees, i.e. ${\rm H}^i({\bf G}/{\bf B},\Ll _{\chi})=0$ for $i>0$. This theorem is ubiquitous in representation theory of algebraic groups in characteristic $p$.
For convenience of the reader, we briefly recall how it can be obtained from the main isomorphism ${\sf F}^n_{\ast}\Ll _{(q-1)\rho} = {\sf St}_q\otimes \Oo _{{\bf G}/{\bf B}}$ (recall that $q=p^n$ for $n\in \mathbb N$). From $\langle \chi ,\alpha ^{\vee}\rangle \geq 0$ one obtains $\langle \chi + \rho ,\alpha ^{\vee}\rangle > 0$ for all simple coroots $\alpha ^{\vee}$. By \cite[Part II, Proposition 4.4]{Jan}, the line bundle $\Ll _{\chi +\rho}$ is ample on ${\bf G}/{\bf B}$. Consider the weight $q(\chi + \rho) - \rho = q\chi + (q-1)\rho$. Since  $\Ll _{\chi +\rho}$ is ample, one can choose $n\in \mathbb N$ large enough so that  the line bundle $\Ll _{q(\chi + \rho)}$ be very ample. From the well--known properties of the Frobenius morphism it then follows 

\vspace*{0.2cm}

\begin{equation}
{\rm H}^i({\bf G}/{\bf B},\Ll _{q\chi + (q-1)\rho})={\rm H}^i({\bf G}/{\bf B},\Ll _{\chi}\otimes {\sf F}^n_{\ast}\Ll _{(q-1)\rho})={\rm H}^i({\bf G}/{\bf B},\Ll _{\chi})\otimes {\sf St}_q.
\end{equation}

\vspace*{0.3cm}

Now the left hand side group vanishes for $i>0$ by Serre's vanishing, the line bundle $\Ll _{q(\chi + \rho)}$ being very ample. Hence, ${\rm H}^i({\bf G}/{\bf B},\Ll _{\chi})=0$ for $i>0$ as well.

\subsection{Derived categories of coherent sheaves}\label{subsec:dercatcohsheaves}

\vspace*{0.1cm}

The content of this section can be found, e.g., in  \cite[Section 1.2, 1.4]{Huyb}.

Let $\sk$ be a field. Assume given a $\sk$--linear triangulated category ${\sf D}$, equipped with a shift functor $[1]\colon {\sf D}\rightarrow {\sf D}$.  For two
objects $A, B \in {\sf D}$ let $\Hom ^{\bullet}_{\sf D}(A,B)$ be
the graded $\sk$-vector space $\oplus _{i\in \mathbb Z}\Hom _{\sf
  D}(A,B[i])$. 
Let ${\sf A}\subset {\sf D}$ be a full triangulated subcategory,
that is a full subcategory of ${\sf D}$ which is closed under shifts and forming distinguished triangles.

\begin{definition}\label{def:orthogonalcat}
The right orthogonal ${\sf A}^{\perp}\subset \sf D$ is defined to be
the full subcategory

\vspace*{0.2cm}

\begin{equation}
{\sf A}^{\perp} = \{B \in {\sf D}\colon \Hom _{\sf D}(A,B) = 0 \}
\end{equation}

\vspace*{0.2cm}

\noindent for all $A \in {\sf A}$. The left orthogonal $^{\perp}{\sf
  A}$ is defined similarly. 

\end{definition}

\begin{definition}\label{def:admissible}
A full triangulated subcategory ${\sf A}$ of ${\sf D}$ is called
{\it right admissible} if the inclusion functor ${\sf A}\hookrightarrow {\sf
  D}$ has a right adjoint. Similarly, ${\sf A}$ is called {\it left
  admissible} if the inclusion functor has a left adjoint. Finally,
${\sf A}$ is {\it admissible} if it is both right and
left admissible.
\end{definition}

If a full triangulated category ${\sf A}\subset {\sf D}$ is right admissible then every object $X\in {\sf D}$ fits into a distinguished triangle
  
\vspace*{0.2cm}
  
\begin{equation}
\dots \longrightarrow  Y\longrightarrow X\longrightarrow Z\longrightarrow Y[1]\rightarrow \dots
\end{equation}

\vspace*{0.2cm}

\noindent with $Y\in {\sf A}$ and $Z\in {\sf A}^{\perp}$. One then
says that there is a semiorthogonal decomposition of ${\sf D}$ into
the subcategories $({\sf A}^{\perp}, \ {\sf A})$. More generally,
assume given a sequence of full triangulated subcategories ${\sf
  A}_1,\dots,{\sf A}_n \subset {\sf D}$. Denote $\langle {\sf
  A}_1,\dots,{\sf A}_n\rangle$ the triangulated subcategory of ${\sf
  D}$ generated by ${\sf A}_1,\dots,{\sf A}_n$.

\begin{definition}\label{def:semdecomposition}
A sequence $({\sf A}_1,\dots,{\sf A}_n)$ of admissible subcategories of
${\sf D}$ is called {\it semiorthogonal} if ${\sf
  A}_i\subset {\sf A}_j^{\perp}$ for $1\leq i < j\leq n$,
and ${\sf A}_i\subset {^{\perp}{\sf A}_j}$ for $1\leq j < i\leq n$.
The sequence $({\sf A}_1,\dots,{\sf A}_n)$ is called a {\it semiorthogonal
  decomposition} of ${\sf D}$ if $\langle {\sf A}_1, \dots, {\sf A}_n
\rangle^{\perp} = 0$, that is ${\sf D} = \langle {\sf A}_1,\dots,{\sf A}_n\rangle$.
\end{definition}

\vspace{0.2cm}

\begin{lemma}\label{lem:admiss_orthogonal}
For a semi--orthogonal decomposition ${\sf D}= \langle {\sf A} ,{\sf B}\rangle$, the subcategory ${\sf A}$ is left admissible and the subcategory ${\sf B}$ is right admissible. Conversely, if ${\sf A}\subset \sf D$ is left (resp. right) admissible, then there is a semi--orthogonal decomposition ${\sf D}=\langle {\sf A}, ^{\perp}{\sf A}\rangle$ (resp. ${\sf D}=\langle {\sf A}^{\perp}, {\sf A}\rangle$).
\end{lemma}

\begin{definition}\label{def:exceptcollection}
An object $E \in \sf D$ of a $\sk$--linear triangulated category $\sf D$ is said to be exceptional if there is an isomorphism of graded $\sk$-algebras 

\vspace{0.2cm}

\begin{equation}
\Hom _{\sf D}^{\bullet}(E,E) = \sk.
\end{equation}

\vspace{0.2cm}

A collection of exceptional objects $(E_0,\dots,E_n)$ in $\sf D$ is called 
exceptional if for $1 \leq i < j \leq n$ one has

\vspace{0.2cm}

\begin{equation}
\Hom _{\sf D}^{\bullet}(E_j,E_i) = 0.
\end{equation}

\vspace{0.2cm}

\end{definition}

Denote by $\langle E_0,\dots,E_n \rangle \subset {\sf D}$ the full
triangulated subcategory generated by the exceptional objects $E_0,\dots,E_n$. One 
proves \cite[Lemma 1.58]{Huyb} that such a category is admissible. \\

Given a smooth algebraic variety $X$ over a field $\sk$, denote  $\Dd ^b(X)$ the bounded derived category of coherent sheaves, and let $\Dd ({\rm QCoh}(X))$ denote 
the unbounded derived category of quasi--coherent sheaves. These are $\sk$--linear triangulated categories. Let $\Ee$ be a vector bundle of rank $r$ on $X$, and consider the associated projective bundle $\pi : \Pp (\Ee)\rightarrow X$. Denote $\Oo _{\pi}(-1)$ the line bundle on $\Pp (\Ee)$ of relative degree $-1$, such that $\pi _{\ast}\Oo _{\pi}(1)=\Ee ^{\ast}$. 
One has \cite[Corollary 8.36]{Huyb}:

\vspace{0.2cm}

\begin{theorem}\label{th:Orvlovth}
The category $\Dd ^b(\Pp (\Ee))$ has a semiorthogonal decomposition: 

\vspace{0.1cm}

\begin{equation}
\Dd ^b(\Pp (\Ee)) = \langle \pi ^{\ast}\Dd ^b(X)\otimes \Oo _{\pi}(-r+1),\dots ,  \pi ^{\ast}\Dd ^b(X)\otimes \Oo _{\pi}(-1),\pi ^{\ast}\Dd ^b(X)\rangle .
\end{equation}

\vspace{0.1cm}

\end{theorem}

\vspace{0.2cm}

We also need some basic facts about generators in triangulated categories  (see \cite{Neem}).

\begin{definition}
Let $\sf D$ be a $\sk$--linear triangulated category. An object $C$ of $\sf D$  is called compact if for any coproduct of objects one has 
$\Hom _{\sf D}(C,\coprod _{\lambda \in \Lambda} X_{\lambda}) = \coprod _{\lambda \in \Lambda}\Hom _{\sf D}(C,X)$.
\end{definition}

\begin{definition}
A $\sk$--linear triangulated category $\sf D$ is called compactly generated if $\sf D$ contains small coproducts, and there exists a small set $\sf T$ of compact objects of $\sf D$, such that $\Hom _{\sf D}({\sf T},X) = 0$ implies $X = 0$. In other words, if $X$ is an object of $\sf D$, and for every $T\in \sf T$ one has $\Hom _{\sf D}(T,X) = 0$, then $X$ must be the zero object.
\end{definition}

\begin{definition}
Let $\sf D$ be a compactly generated triangulated category. A set $\sf T$ 
of compact objects of $\sf D$ is called a generating set if $\Hom _{\sf D}({\sf T},X)=0$ implies $X=0$ and $\sf T$ is closed under the shift functor, i.e. ${\sf T} = {\sf T}[1]$.
\end{definition}

\begin{definition}
Let $X$ be a quasi-compact, separated scheme. An object $C\in \Dd ({\rm QCoh}(X))$ is called perfect if, locally on $X$, it is isomorphic to a bounded complex of locally free sheaves of finite type.
\end{definition}

\begin{proposition}{\cite[Example 1.10]{Neem}}\label{prop:ample_line_bundle_gen_set}
Let $X$ be a quasi--compact, separated scheme, and $\Ll$ be an ample line bundle on $X$. Then the set $\langle \Ll ^{\otimes m}[n]\rangle, m,n\in \mathbb Z$ is a generating set for $\Dd ({\rm QCoh}(X))$.
\end{proposition}

Finally, recall that given two smooth varieties $X$ and $Y$ over $\sk$, an object 
$\mathcal P \in \Dd ^b(X\times Y)$ gives rise to an integral transform $\Phi _{\mathcal P}(-): = {\pi _Y}_{\ast}(\pi _{X}^{\ast}(-)\otimes \mathcal P)$ between 
$\Dd ^b(X)$ and $\Dd ^b(Y)$, where $\pi _{X}, \pi _{Y}$ are the projections of $X\times Y$ onto corresponding factors.

\begin{proposition}{\cite[Proposition 5.1]{Huyb}}\label{prop:F-M_composition}
Let $X, Y,$ and $Z$ be smooth projective varieties over a field $\sk$. Consider objects $\mathcal P\in \Dd ^b(X\times Y)$ and $\mathcal Q\in \Dd ^b(Y\times Z)$. Define the object $\mathcal R\in \Dd ^b(X\times Z)$ by the formula
${\pi _{XZ}}_{\ast}(\pi _{XY}^{\ast}\mathcal P\otimes \pi _{YZ}^{\ast}\mathcal Q)$, where $\pi _{XZ}, \pi _{XY}$, and $\pi _{YZ}$ are the projections from 
$X\times Y\times Z$ to $X\times Z$ (resp., to $X\times Y$, resp., to $Y\times Z$).
Then the composition $\Phi _{\mathcal Q}\circ \Phi _{\mathcal P}: \Dd ^b(X)\rightarrow \Dd ^b(Z)$ is isomorphic to the integral transform $\Phi _{\mathcal R}$.
\end{proposition}

\vspace{0.3cm}


\section{\bf Semiorthogonal decompositions for flag varieties}\label{sec:main}


\vspace*{0.3cm}

In order to prove Lemma \ref{lem:mainlemma} below, the key statement of this section, we need an auxiliary proposition which is a derived category counterpart of the main theorem of \cite{CPS}.

\begin{proposition}\label{prop:derverDemazure_char_for}
Let $\pi : {\bf G}/{\bf B}\rightarrow {\rm Spec}(\sk)$ be the structure morphism, and for a simple root $\alpha _i$ denote $\pi _{\alpha _i}: {\bf G}/{\bf B}\rightarrow {\bf G}/{\bf P}_{\alpha _i}$ the projection, a $\Pp ^1$--bundle over ${\bf G}/{\bf P}_{\alpha _i}$. Let $w_0$ be the longest element of $\mathcal W$, and let $s_{\alpha _1}s_{\alpha _2}\dots \cdot s_{\alpha _N}$ be a reduced expression of $w_0$. Then there is an isomorphism of functors:

\vspace*{0.2cm}

\begin{equation}\label{eq:WCF_dercat}
\pi ^{\ast}\pi _{\ast} = \pi _{\alpha _N}^{\ast}{\pi _{\alpha _N}}_{\ast}\pi _{\alpha _{N-1}}^{\ast}{\pi _{\alpha _{N-1}}}_{\ast}\dots \pi _{\alpha _1}^{\ast}{\pi _{\alpha _1}}_{\ast} .
\end{equation}

\vspace*{0.2cm}

\end{proposition}

\begin{proof}
Denote $\mathcal Z$ the fibered product ${\bf G}/{\bf B}\times _{{\bf G}/{\bf P}_{\alpha _1}}{\bf G}/{\bf B}\times \dots \times _{{\bf G}/{\bf P}_{\alpha _N}}{\bf G}/{\bf B}$, and let $p:\mathcal Z\rightarrow {\bf G}/{\bf B}\times {\bf G}/{\bf B}$ denote the projection onto the two extreme factors. Then, by Proposition \ref{prop:F-M_composition}, the functor in the right hand side of (\ref{eq:WCF_dercat}) is given by an integral transform whose kernel is isomorphic to the direct image $p_{\ast}\Oo _{\mathcal Z}\in \Dd ^b({\bf G}/{\bf B}\times {\bf G}/{\bf B})$. Observe that $\mathcal Z$ is isomorphic to ${\bf G}\times _{\bf B}Z_{w_0}$, where $Z_{w_0}: = {\bf P}_{\alpha _1}\times \dots \times {\bf P}_{\alpha _N}/{\bf B}^N$ is the Demazure variety corresponding to the reduced expression of $w_0$ as above. Indeed, by definition of these varieties \cite[Definition 2.2.1]{BK} and [Diagram ($\mathcal D$), p.66] of {\it loc.cit.}, one has an isomorphism 
$({\bf G}\times _{\bf B}Z_{w_0s_{\alpha _N}})\times _{{\bf G}/{\bf P}_{\alpha _N}}{\bf G}/{\bf B}= {\bf G}\times _{\bf B}(Z_{w_0s_{\alpha _N}}\times _{{\bf G}/{\bf P}_{\alpha _N}}{\bf G}/{\bf B})={\bf G}\times _{\bf B}Z_{w_0}$.

The projection $p$ maps ${\mathcal Z}={\bf G}\times _{\bf B}Z_{w_0}$ onto ${\bf G}/{\bf B}\times {\bf G}/{\bf B}$. Consider the base change of $p$ along the quotient morphism $q: {\bf G}\times {\bf G}/{\bf B}\rightarrow  {\bf G}/{\bf B}\times  {\bf G}/{\bf B}$: one obtains the projection $p':{\bf G}\times Z_{w_0}\rightarrow {\bf G}\times {\bf G}/{\bf B}$ that factors as ${\rm id}\times d$, where $d: Z_{w_0}\rightarrow {\bf G}/{\bf B}$ is the projection. By Proposition \ref{prop:Demazure_vanishing} below, $p'_{\ast}\Oo _{{\bf G}\times Z_{w_0}}=\Oo _{{\bf G}\times {\bf G}/{\bf B}}$. Since the quotient morphism ${\bf G}\rightarrow {\bf G}/{\bf B}$ is flat, by flat base change one obtains $q^{\ast}p_{\ast}\Oo _{\mathcal Z}=p'_{\ast}\Oo _{{\bf G}\times Z_{w_0}}=\Oo _{{\bf G}\times {\bf G}/{\bf B}}$. Applying $q_{\ast}$ to $q^{\ast}p_{\ast}\Oo _{\mathcal Z}$ and using the projection formula, one arrives at an isomorphism $q_{\ast}(q^{\ast}p_{\ast}\Oo _{\mathcal Z})=p_{\ast}\Oo _{\mathcal Z}\otimes q_{\ast}\Oo _{{\bf G}\times {\bf G}/{\bf B}}=q_{\ast}\Oo _{{\bf G}\times {\bf G}/{\bf B}}$. It follows that $p_{\ast}\Oo _{\mathcal Z}$ is an invertible sheaf on ${\bf G}/{\bf B}\times {\bf G}/{\bf B}$ isomorphic to $\Ll \boxtimes \Oo _{{\bf G}/{\bf B}}$ for some line bundle $\Ll$ on ${\bf G}/{\bf B}$. Applying the integral transform $\Phi _{w_0}=\Phi _{w_0^{-1}}=\Phi _{p_{\ast}\Oo _{\mathcal Z}}$ to $\Oo _{{\bf G}/{\bf B}}$, one obtains $\Phi _{w_0^{-1}}(\Oo _{{\bf G}/{\bf B}})=\pi _{\alpha _1}^{\ast}{\pi _{\alpha _1}}_{\ast}\pi _{\alpha _{2}}^{\ast}{\pi _{\alpha _{2}}}_{\ast}\dots \pi _{\alpha _N}^{\ast}{\pi _{\alpha _N}}_{\ast}(\Oo _{{\bf G}/{\bf B}}) = 
\Oo _{{\bf G}/{\bf B}}=\Phi _{\Oo _{{\bf G}/{\bf B}}\boxtimes \Ll}(\Oo _{{\bf G}/{\bf B}})=\pi _{\ast}\Oo _{\bf G/B}\otimes 
\Ll ={\mathbb H}^{\cdot}({\bf G/B},\Oo _{\bf G/B})\otimes \Ll=\Ll$, where the last isomorphism follows from Corollary \ref{cor:admisscor} below. Therefore, $p_{\ast}\Oo _{\mathcal Z} = \Oo _{{\bf G}/{\bf B}\times {\bf G}/{\bf B}}$. Finally, by flat base change for the morphism $\pi : {\bf G}/{\bf B}\rightarrow {\rm Spec}(\sk)$ along itself, the integral transform $\Phi _{\Oo _{{\bf G}/{\bf B}\times {\bf G}/{\bf B}}}$ is isomorphic to $\pi ^{\ast}\pi _{\ast}$.
\end{proof}

\begin{proposition}\label{prop:Demazure_vanishing}
Let $d: Z_{w_0}\rightarrow {\bf G}/{\bf B}$ be the projection map as above. Then $d_{\ast}\Oo _{Z_{w_0}}=\Oo _{{\bf G}/{\bf B}}$.
\end{proposition}

\begin{proof}

Given an element $w\in \mathcal W$, denote the corresponding Schubert variety by $X_w$, and let $d_w: Z_w= {\bf P}_{\alpha _1}\times \dots \times {\bf P}_{\alpha _n}/{\bf B}^n\rightarrow X_w$ denote the Demazure desingularization. Then $d=d_{w_0}: Z_{w_0}\rightarrow {\bf G}/{\bf B}$ is a birational morphism onto ${\bf G}/{\bf B}$, since $X_{w_0}={\bf G}/{\bf B}$. The flag variety being smooth, hence normal, by Zariski's main theorem one has ${\rm R}^0d_{\ast}\Oo _{Z_{w_0}}=\Oo _{{\bf G}/{\bf B}}$. To prove the vanishing of higher direct images ${\rm R}^id_{\ast}\Oo _{Z_{w_0}}$ for $i>0$, one can argue as in the \cite[Theorem 3.3.4, (b)]{BK}. More specifically,
one argues by induction on the length $l(w)$ of an element $w\in \mathcal W$ to prove that ${\rm R}^i{d_w}_{\ast}\Oo _{Z_w}=0$ for $i\geq 1$; if $l(w)=1$ then 
$d_w$ is an isomorphism. Given a reduced expression $s_{\alpha _1}\dots s_{\alpha _n}$ of $w$, set $v=s_{\alpha _2}\dots s_{\alpha _n}$, and 
consider the factorization of $d_w$ as $d _{\alpha _1v}: Z_w={\bf P}_{\alpha _1}\times _{\bf B}Z_{v}\rightarrow {\bf P}_{\alpha _1}\times _{\bf B}X_{v}$ followed by the product morphism $f: {\bf P}_{\alpha _1}\times _{\bf B}X_{v}\rightarrow X_{w}$. Then, by induction one obtains ${\rm R}^i{d_v}_{\ast}\Oo _{Z_v}=0$ for $i\geq 1$ which implies ${{\rm R}^id _{\alpha _1v}}_{\ast}\Oo _{Z_w}=0$ for $i\geq 1$ as well. Finally, \cite[Proposition 3.2.1, (b)]{BK} implies that for the product morphism $f: {\bf P}\times _{\bf B}X_v\rightarrow {\bf P}X_v$, where $\bf P$ is the minimal parabolic subgroup corresponding to a simple reflection,
the higher direct images ${\rm R}^if_{\ast}\Oo _{{\bf P}\times _{\bf B}X_v}$ are trivial for $i\geq 1$. Hence, for the composed morphism $d_w$ the higher direct images ${\rm R}^i{d _{w}}_{\ast}\Oo _{Z_w}=0$ are trivial for $i\geq 1$.
\end{proof}

\begin{corollary}\label{cor:admisscor}
One has ${\rm H}^i({\bf G/B},\Oo _{\bf G/B})=0$ for $i>0$. Given a line bundle $\Ll$ on ${\bf G}/{\bf B}$, the triangulated subcategory $\langle \Ll\rangle$ of $\Dd ^b(\bf G/B)$ generated by $\Ll$ is admissible.
\end{corollary}

\begin{proof}
By Proposition \ref{prop:Demazure_vanishing}, one has $d_{\ast}\Oo _{X_{w_0}}=\Oo _{\bf G/B}$. On the other hand, by its construction, the variety $X_{w_0}$ is isomorphic to an iterated sequence of $\Pp ^1$--bundles over a point; hence, ${\rm H}^i(X_{w_0},\Oo _{X_{w_0}})=0$ for $i>0$. Therefore, ${\rm  H}^i({\bf G/B},\Oo _{\bf G/B})={\rm  H}^i({\bf G/B},d_{\ast}\Oo _{\mathcal  X})={\rm H}^i(X_{w_0},\Oo _{X_{w_0}})=0$ for $i>0$.
It follows from Section \ref{subsec:dercatcohsheaves} that the category $\langle \Ll\rangle$ is admissible once the bundle $\Ll$ is exceptional, i.e. $\Hom _{\bf G/B}^{\bullet}(\Ll ,\Ll)=k$. The latter condition is equivalent to ${\rm H}^i({\bf G/B},\Oo _{\bf G/B})=0$ for $i>0$.
\end{proof}

\begin{lemma}\label{lem:mainlemma}
Consider the semiorthogonal decomposition of $\Dd ^b(\bf G/B) = \langle \langle \Oo _{{\bf G}/{\bf B}}\rangle ^{\perp},\langle  \Oo _{{\bf G}/{\bf B}}\rangle \rangle$. Then the subcategory $\langle \Oo _{{\bf G}/{\bf B}}\rangle^{\perp}\subset \Dd ^b(\bf G/B)$ is generated, as an admissible triangulated subcategory of $\Dd ^b({\bf G}/{\bf B})$, by acyclic line bundles $\Ll _{\chi}$ with the following property: there exists a simple coroot $\alpha ^{\vee}\in {\rm R}^{\vee}$, such that $\langle \chi + \rho, \alpha ^{\vee}\rangle =0$.
\end{lemma}

\begin{remark}
{\rm The generating set of the subcategory $\langle \Oo _{{\bf G}/{\bf B}}\rangle ^{\perp}$ in Lemma \ref{lem:mainlemma} is not at all minimal.
}
\end{remark}

\begin{proof}
By Corollary \ref{cor:admisscor}, the category $\langle  \Oo _{{\bf G}/{\bf B}}\rangle$ is admissible, hence its right orthogonal is an admissible subcategory of $\Dd ^b(\bf G/B)$.\\

Given a simple root $\alpha$, consider the corresponding minimal parabolic subgroup ${\bf P}_{\alpha}$, and let $\pi _{\alpha}:{\bf G}/{\bf B}\rightarrow {\bf G}/{\bf P}_{\alpha}$ denote the projection.
Observe first that $\langle \Oo _{{\bf G}/{\bf B}}\rangle ^{\perp}$ contains the subcategory generated by $\pi _{\alpha}^{\ast}{\pi _{\alpha}}_{\ast}\Ff \otimes \Ll _{-\rho}$, where $\Ff \in \Dd ^b({\bf G}/{\bf B})$. Indeed, 

\begin{equation}
\Hom  _{{\bf G}/{\bf B}}^{\bullet}(\Oo _{{\bf G}/{\bf B}},\pi _{\alpha}^{\ast}{\pi _{\alpha}}_{\ast}\Ff \otimes \Ll _{-\rho})  = {\mathbb H}^{\ast}({\bf G}/{\bf P}_{\alpha},{\pi _{\alpha}}_{\ast}\Ff\otimes {\pi _{\alpha}}_{\ast}\Ll _{-\rho}) =0,
\end{equation}

\vspace{0.2cm}

as ${\pi _{\alpha}}_{\ast}\Ll _{-\rho}=0$. Let ${\mathcal C}\subset \Dd ^b({\bf G}/{\bf B})$ be the full triangulated category generated by $\pi _{\alpha}^{\ast}{\pi _{\alpha}}_{\ast}\Ff \otimes \Ll _{-\rho}$, where $\Ff \in \Dd ^b({\bf G}/{\bf P}_{\alpha})$ and $\alpha$ runs over the set of all the simple roots.
Observe next that ${\mathcal C}$ coincides with the triangulated subcategory generated by line bundles satisfying the condition of Lemma \ref{lem:mainlemma}: one the one hand, given a simple root $\alpha$, any line bundle $\Ll _{\chi}$ on ${\bf G}/{\bf P}_{\alpha}$ satisfies $\langle \chi,\alpha ^{\vee}\rangle =0$, and the projection functor ${\pi _{\alpha}}_{\ast}:\Dd ^b({\bf G}/{\bf B})\rightarrow 
\Dd ^b({\bf G}/{\bf P}_{\alpha})$ is surjective. Hence, ${\mathcal C}\supset \langle \Ll _{\chi}\rangle$ with $\langle \chi + \rho, \alpha ^{\vee}\rangle =0$ for a simple coroot $\alpha ^{\vee}$. On the other hand, upon choosing an ample line bundle $\Ll$ on ${\bf G}/{\bf P}_{\alpha}$, the category $\Dd ^b({\bf G}/{\bf P}_{\alpha})$ is generated by the set $\langle \Ll ^{\otimes m}[n]\rangle, m,n\in \mathbb Z$ in virtue of Proposition \ref{prop:ample_line_bundle_gen_set}, and one obtains a converse inclusion ${\mathcal C}\subset \langle \Ll _{\chi}\rangle$ with
$\chi$ as in the statement of the lemma.\\

Consider its left orthogonal $^{\perp}{\mathcal C}\subset \Dd ^b({\bf G}/{\bf B})$.
By Lemma \ref{lem:admiss_orthogonal}, it is an admissible subcategory of $\Dd ^b({\bf G}/{\bf B})$. The same lemma implies that the statement of Lemma \ref{lem:mainlemma} is equivalent to saying that the category $^{\perp}{\mathcal C}$ is equivalent to $\langle \Oo _{{\bf G}/{\bf B}}\rangle$. To this end, observe that any object $\mathcal G$ of $^{\perp}{\mathcal C}$ 
belongs to $\pi _{\alpha}^{\ast}\Dd ^b({\bf G}/{\bf P}_{\alpha})$ for each simple root $\alpha$; in other words, $^{\perp}{\mathcal C}\subset \bigcap _{\alpha}\pi _{\alpha}^{\ast}\Dd ^b({\bf G}/{\bf P}_{\alpha})$. Indeed, by Serre duality

\vspace{0.2cm}

\begin{equation}
\Hom  _{{\bf G}/{\bf B}}^{\bullet}(\mathcal G,\pi _{\alpha}^{\ast}{\pi _{\alpha}}_{\ast}\Ff \otimes \Ll _{-\rho}) = \Hom  _{{\bf G}/{\bf B}}^{\bullet}({\pi _{\alpha}}_{\ast}\Ff,{\pi _{\alpha}}_{\ast} (\mathcal G\otimes \Ll _{-\rho})[{\rm dim}({\bf G}/{\bf B}])^{\ast}=0;
\end{equation}

\vspace{0.2cm}

since $\Ff$ is arbitrary and functor ${\pi _{\alpha}}_{\ast}$ is surjective, it follows that ${\pi _{\alpha}}_{\ast} (\mathcal G\otimes \Ll _{-\rho})=0$. By Theorem \ref{th:Orvlovth}, this implies $\mathcal G\in \pi _{\alpha }^{\ast}\Dd ^b({\bf G}/{\bf P}_{\alpha})$, the line bundle $\Ll _{-\rho}$ having degree $-1$ along $\pi _{\alpha}$.

Let $\sf M$ be an object $^{\perp}{\mathcal C}$. By the above, $\sf M = \pi ^{\ast}_{\alpha}{\pi _{\alpha}}_{\ast}\sf M$ for any simple root $\alpha$. By Proposition 
\ref{prop:derverDemazure_char_for}, one has an isomorphism 
$\pi ^{\ast}\pi _{\ast}{\sf M} = \pi _{\alpha _N}^{\ast}{\pi _{\alpha _N}}_{\ast}\pi _{\alpha _{N-1}}^{\ast}{\pi _{\alpha _{N-1}}}_{\ast}\dots \pi _{\alpha _1}^{\ast}{\pi _{\alpha _1}}_{\ast}{\sf M}={\sf M}$, hence ${\sf M}\in \langle \Oo _{{\bf G}/{\bf B}}\rangle$.
\end{proof}


\begin{corollary}\label{cor:maincorollary}
Consider the semiorthogonal decomposition of $\Dd ^b(\bf G/B) = \langle \langle \Ll _{-\rho}\rangle ,^{\perp} \langle \Ll _{-\rho}\rangle \rangle$. Then the category $^{\perp} \langle \Ll _{-\rho}\rangle$ is generated, as an admissible triangulated subcategory of $\Dd ^b({\bf G}/{\bf B})$, by the set of line bundles $\Ll _{\chi}$, 
where $\langle \chi,\alpha ^{\vee}\rangle =0$ for some simple coroot $\alpha ^{\vee}\in {\rm R}^{\vee}$. 
\end{corollary}

\begin{proof}
 Let $\Ee \in $$^{\perp}\langle \Ll _{-\rho}\rangle$, then by Serre duality 

\begin{eqnarray}
& \Hom _{{\bf G}/{\bf B}}^{\bullet}(\Ee,\Ll _{-\rho})= \Hom  _{{\bf G}/{\bf B}}^{\bullet}(\Ll _{-\rho},\Ee \otimes \Ll _{-2\rho}[{\rm dim}({\bf G}/{\bf B}]))^{\ast} = \\
& \Hom  _{{\bf G}/{\bf B}}^{\bullet}(\Oo _{{\bf G}/{\bf B}},\Ee \otimes \Ll _{-\rho}[{\rm dim}({\bf G}/{\bf B}]))^{\ast}=0.\nonumber
\end{eqnarray}

\vspace{0.2cm}

Therefore, up to a twist by the line bundle $\Ll _{\rho}$, the category $^{\perp} \langle \Ll _{-\rho}\rangle$  is equivalent to the subcategory $\langle \Oo _{{\bf G}/{\bf B}}\rangle ^{\perp}$.  Lemma \ref{lem:mainlemma} implies the statement.
\end{proof}

\vspace*{0.3cm}

\section{\bf The Steinberg line bundle}

\vspace*{0.3cm}

Consdier the admissible subcategory $\langle \Ll _{-\rho}\rangle$ of $\Dd ^b(\bf G/B)$.
It follows from the above that the isomorphism ${\sf F}^n_{\ast}\Ll _{-\rho} = {\sf St}_q\otimes \Ll _{-\rho}$ is equivalent to the following statement:

\begin{theorem}\label{th:Haboushth}
One has ${\sf F}^n_{\ast}\Ll _{-\rho}\in \langle \Ll _{-\rho}\rangle$.
\end{theorem}

\begin{proof}
By Corollary \ref{cor:maincorollary}, the fact that the bundle ${\sf F}^n_{\ast}\Ll _{-\rho}$ belongs to the subcategory $\langle \Ll _{-\rho}\rangle\subset \Dd ^b(\bf G/B)$ is equivalent to saying that ${\sf F}^n_{\ast}\Ll _{-\rho}$ is right orthogonal to the subcategory 
$\langle \Ll _{\chi}\rangle$ generated by all $\Ll _{\chi}$, where $\langle \chi,\alpha ^{\vee}\rangle =0$ for some simple coroot $\alpha ^{\vee}\subset {\rm R}^{\vee}$.  In other words, one has to ensure that

\begin{equation}\label{eq:Steinbergisom}
\Hom ^{\bullet}_{\bf G/B}( \Ll _{\chi},{\sf F}^n_{\ast}\Ll _{-\rho})={\mathbb H}^{\ast}({\bf G/B},\Ll _{-p^n\chi - \rho})=0.
\end{equation}

\vspace{0.2cm}

By Remark \ref{rem:Demazure_bits_over_Z}, the line bundle $\Ll _{\mu}$ is acyclic if $\langle \mu +\rho,\alpha ^{\vee}\rangle = 0$ for some simple coroot $\alpha ^{\vee}$. Taking $\mu = -p^n\chi - \rho$, one obtains $\langle \mu +\rho,\alpha ^{\vee}\rangle =
\langle -p^n\chi - \rho +\rho,\alpha ^{\vee}\rangle = -p^n\langle \chi,\alpha ^{\vee}\rangle =0$. Hence, the bundle $\Ll _{-p^n\chi - \rho}$ is acyclic, and (\ref{eq:Steinbergisom}) holds.
\end{proof}

\end{document}